\documentclass[12pt]{amsart}

\usepackage{amsthm}

\usepackage{amssymb}
\usepackage{verbatim}
\usepackage[curve,matrix,arrow]{xy}
\usepackage{amsmath,amsfonts}
\usepackage{graphicx}
\usepackage{epsf}

\usepackage{floatflt}

\usepackage{amssymb, graphics}

\newcommand{\mathsym}[1]{{}}

\newcommand{\thmref}[1]{Theorem~\ref{#1}}

\newcommand{\eqnref}[1]{Equation~(\ref{#1})}

\newcommand{\figref}[1]{Figure~\ref{#1}}

  {\end{list}}

\newtheorem{theorem}{Theorem}[section]

\theoremstyle{definition}

\newcommand{\ga}{\Gamma}

\newcommand{\vv}[1]{V(#1)}
\newcommand{\va}{\upsilon}

\def\can{{\mathop{\rm can}}}

\def\<{\langle }
\def\>{\rangle }
\newcommand{\secref}[1]{\S\ref{#1}}

\begin{document}

\title[Effective Resistances and Invariants of Ladder Graphs]{Effective Resistances, Kirchhoff index and Admissible Invariants of Ladder Graphs}

\author{Zubeyir Cinkir}
\address{Zubeyir Cinkir\\
Department of Industrial Engineering\\
Abdullah G\"{u}l University\\
Kayseri\\
TURKEY.}
\email{zubeyirc@gmail.com}




\begin{abstract}
We explicitly compute the effective resistances between any two vertices of a ladder graph by using circuit reductions. Using our findings, we obtain explicit formulas for Kirchhoff index and admissible invariants of a ladder graph considering it as a model of a metrized graph.
Comparing our formula for Kirchhoff index and previous results in literature, we obtain an explicit sum formula involving trigonometric functions. We also expressed our formulas in terms of certain generalized Fibonacci numbers that are the values of the Chebyshev polynomials of the second kind at $2$.
\end{abstract}

\maketitle

\section{Introduction}\label{sec introduction}

\begin{floatingfigure}[r]{2.4 in}
\begin{center}
\includegraphics[scale=0.55]{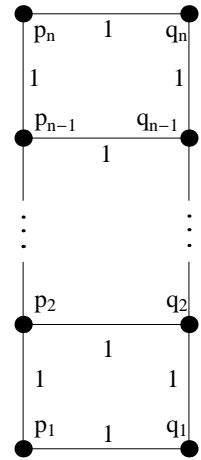}
\end{center}
\caption{Ladder graph $L_n$ with $2n$ vertices.} \label{fig laddergraphn}
\end{floatingfigure}
A ladder graph $L_n$ is a planar graph that looks like a ladder with $n$ rungs as shown in \figref{fig laddergraphn}.
It has $2n$ vertices and $3n-2$ edges. Each of its edges has length $1$, so the total length of $L_n$ is $\ell(L_n):=3n-2$.
We label the vertices on the right and left as $\{ q_1, \, q_2, \, \cdots, q_n \}$ and $\{ p_1, \, p_2, \, \cdots, p_n \}$, respectively.

One can consider $L_n$ as an electrical network in which the resistances along edges are given by the corresponding edge lengths.
For the ladder graph $L_n$, Kirchhoff index and resistance values between vertices are studied in \cite{CEM} by using the spectral properties of the discrete Laplacian of $L_n$, and closed form formulas are obtained in terms of Chebyshev polynomials.

In this paper, we obtained explicit formulas for Kirchhoff index and resistances between vertices of $L_n$ with a rather elementary method.
Namely, we used circuit reductions and solved a number of recurrence relations. Moreover, by considering $L_n$ as a model of a metrized graph, we derived explicit formulas for its admissible invariants considered in \cite{C2}, \cite{C3}, \cite{C4}, \cite{Zh} and the references therein. At the end, we expressed these formulas in terms of a sequence of generalized Fibonacci numbers $G_n$ defined by $G_{n+2}=4G_{n+1}-G_n$ if $n \geq 2$, $G_1=1$ and $G_0=0$.
The number $G_n$ is known to be the number of spanning trees in $L_n$, and that $G_n=U_{n-1}(2)$, where $U_n(x)$ is the Chebyshev polynomial of the second kind.

Among other things, we showed that the Kirchhoff index of $L_n$ satisfies the following equalities for each positive integer $n$ (see \thmref{thm Kirchhoff index} and \eqnref{eqn KI Ladder II} below):
\begin{equation*}\label{eqn KI Ladder II0}
\begin{split}
Kf(L_n) &=\frac{n^3}{3}+ \frac{n^2 G_{2n}}{6 G_{n}^2}\\
&=\frac{n^3}{3}-\frac{n^2}{\sqrt{3}}\Big[ 1- \frac{2}{1-(2-\sqrt{3})^{2n}}\Big].
\end{split}
\end{equation*}
and we derived the following trigonometric sum formulas (see \eqnref{eqn KI Ladder II} and \eqnref{eqn KI Ladder IV} below):
\begin{equation*}\label{eqn KI Ladder IV0}
\begin{split}
\sum_{k=0}^{n-1} \frac{1}{1+2 \sin^2{(\frac{k \pi}{2n})}}=\frac{1}{3}+ \frac{n G_{2n}}{6 G_{n}^2} \qquad \text{and } \quad
 \sum_{k=1}^{n-1} \frac{1}{ \sin^2{(\frac{k \pi}{2n})}}=\frac{2(n^2-1)}{3}.
\end{split}
\end{equation*}

The resistance values on Wheel and Fan graphs are expressed in terms of generalized Fibonacci numbers in \cite{BG}. Our findings for resistance values on a Ladder graph are analogues of those results on Wheel and Fan graphs.


\section{Resistances between any pairs of vertices in $L_n$}\label{sec resistances}

Let $r(p,q)$ be the effective resistance between the vertices $p$ and $q$ in $L_n$. We also use the notation $r_{L_n}(p,q)$ for this value to emphasize the graph the resistance being computed in.
In this section, we find explicit formula of $r(p,q)$ for every pair of vertices $p$ and $q$ of $L_n$. Using the symmetry of the graph $L_n$, for all $i, \, j \in \{1,\, 2, \, \cdots, n \}$ we have
\begin{equation}\label{eqn symmetry}
\begin{split}
r(p_i,p_j)=r(q_i,q_j), \quad \text{and} \quad  r(p_i,q_j)=r(q_i,p_j).
\end{split}
\end{equation}
First, we compute
effective resistances between the end vertices $p_1$, $p_n$, $q_1$ and $q_n$. Set
$x_n:=r_{L_n}(p_n,p_1)$, $y_n:=r_{L_n}(p_n,q_1)$ and $z_n:=r_{L_n}(p_n,q_n)$.

\begin{floatingfigure}[r]{3.7 in}
\begin{center}
\includegraphics[scale=0.6]{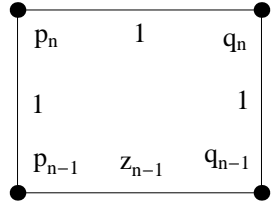}
\end{center}
\caption{Ladder graph $L_n$ with circuit reduction of $L_{n-1}$ with respect to $p_{n-1}$ and $q_{n-1}$, where $n \geq 2$.} \label{fig laddergraph2}
\end{floatingfigure}
Suppose we make circuit reduction of $L_{n-1}$ with respect to the vertices $p_{n-1}$ and $q_{n-1}$. Since we obtain $L_n$ by adding the vertices $p_n$ and $q_n$, and the three edges with end points $\{ p_{n-1}, p_{n} \}$, $\{ p_{n}, q_{n} \}$ and $\{ q_{n}, q_{n-1} \}$, we have the circuit reduction of $L_n$ as shown in \figref{fig laddergraph2}. Now, using the parallel circuit reduction in this graph, we can express $z_n$ in terms of $z_{n-1}$. This gives us the following recurrence relation:
\begin{equation}\label{eqn zn recurrence}
\begin{split}
z_n &= \frac{z_{n-1}+2}{z_{n-1}+3}, \quad \text{for all $n \geq 2$}.\\
z_1 &=1.
\end{split}
\end{equation}
Now, we use Mathematica \cite{MMA} to solve this recurrence relation.
This gives
\begin{equation}\label{eqn zn recurrence2}
\begin{split}
z_n = -1-\sqrt{3}+\frac{2 \sqrt{3}}{1-(2-\sqrt{3})^{2n}}, \quad \text{for all $n \geq 1$},
\end{split}
\end{equation}
which indeed the solution of \eqnref{eqn zn recurrence}. In particular, we have
$z_1=1$, $z_2=\frac{3}{4}$, $z_3=\frac{11}{15}$, $z_4=\frac{41}{56}$, $z_5=\frac{153}{209}$, $z_6=\frac{571}{780}$.

Other equivalent forms of $z_n$ can be given as follows:
\begin{equation}\label{eqn zn others}
\begin{split}
z_n=-1-\sqrt{3}+\frac{2 \sqrt{3}(2+\sqrt{3})^{n}}{(2+\sqrt{3})^{n}-(2-\sqrt{3})^{n}}, \quad \text{or} \quad
z_n=-1-\sqrt{3}\coth \big(n \ln(2-\sqrt{3}) \big),
\end{split}
\end{equation}
where $\coth$ is the hyperbolic cotangent function.
Note that $(2-\sqrt{3})(2+\sqrt{3})=1$.

We can rewrite \eqnref{eqn zn recurrence} in the following form:
$$z_n=\frac{1}{1+\frac{1}{2+z_{n-1}}},$$
and if we use this equality to express $z_{n-1}$ in terms of $z_{n-2}$ and substitute it in this equality, we obtain
$$z_n=\frac{1}{1+\frac{1}{2+\frac{1}{1+\frac{1}{2+z_{n-2}}}}}.$$
We can repeat this process to express $z_n$ in terms of $z_k$ for any positive integer $k <n$. Since $0< z_n <1$ for each integer $n \geq 2$ and $z_n$ is decreasing by \eqnref{eqn zn recurrence}, we notice that $z_n$'s must be part of the convergents of the number with continued fraction expansion $[ 0,1,2,1,2,1,2,\cdots ]$. On the other hand, this is nothing but the every other terms in the continued fraction expansion of $\sqrt{3}-1$. Probabilistic explanation of these facts via spanning trees can be found in \cite[page 11]{LP}.

This kind of circuit reduction technique that we used to find $z_n$ was used in the case of infinite ladder in \cite[Chapter 22-Section 6]{FLS}.

Our next aim is to find explicit formulas for $x_n$ and $y_n$ as we did for $z_n$.

\begin{floatingfigure}[r]{3.2 in}
\begin{center}
\includegraphics[scale=0.7]{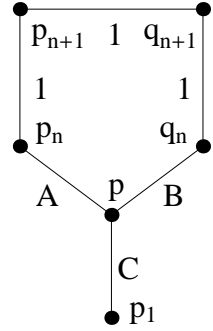}
\end{center}
\caption{Ladder graph $L_{n+1}$ with circuit reduction of $L_{n}$ with respect to $p_{n}$, $q_{n}$ and $p_1$, where $n \geq 1$.} \label{fig laddergraph3}
\end{floatingfigure}

Now, suppose $n \geq 1$ and we make circuit reduction of the subgraph $L_{n}$ of $L_{n+1}$ with respect to the vertices $p_{n}$, $q_{n}$ and $p_1$. That is, the part $L_{n}$ in $L_{n+1}$ is reduced to a $Y$-shaped graph with the outer vertices $p_{n}$, $q_{n}$ and $p_1$, and having the effective resistances $A$, $B$ and $C$ between the end points of its edges. This is illustrated in \figref{fig laddergraph3}. Then we have
$B+C=y_n$,  $A+C=x_n$ and   $A+B=z_n.$
Solving these gives
$A=\frac{x_n -y_n +z_n}{2}$, $B=\frac{-x_n +y_n +z_n}{2}$ and $C=\frac{x_n +y_n -z_n}{2}.$
On the other hand, using parallel and series circuit reductions in \figref{fig laddergraph3} we obtain
$x_{n+1}=\frac{(A+1)(B+2)}{z_n +3}+C$ and $y_{n+1}=\frac{(B+1)(A+2)}{z_n +3}+C$. Therefore,
\begin{equation}\label{eqn recurrence xnyn}
\begin{split}
x_{n+1} &= \frac{(x_n-y_n+z_n+2)(-x_n+y_n+z_n+4)}{4(z_n+3)}+\frac{x_n+y_n-z_n}{2}, \quad \text{if $n \geq 1$}.\\
y_{n+1} &= \frac{(-x_n+y_n+z_n+2)(x_n-y_n+z_n+4)}{4(z_n+3)}+\frac{x_n+y_n-z_n}{2}, \quad \text{if $n \geq 1$}.\\
 x_1 &=0 \qquad \text{and } \qquad y_1=1.
\end{split}
\end{equation}
If we subtract the second equation from the first one, we obtain $x_{n+1}-y_{n+1}=\frac{x_n -y_n}{z_n +3}$. Now, we set $t_n:=x_n-y_n$ to obtain
\begin{equation}\label{eqn tn}
\begin{split}
t_{n+1}=\frac{t_n}{z_n +3}, \qquad \text{if $n \geq 1$ and $t_1=-1$}.
\end{split}
\end{equation}
This can be rewritten as follows
\begin{equation}\label{eqn tn2}
\begin{split}
t_{n+1}=-\prod_{k=1}^n \frac{1}{z_k +3}.
\end{split}
\end{equation}
Since $\frac{1}{z_k +3}=\frac{\left(2+\sqrt{3}\right)^k-\left(2-\sqrt{3}\right)^k}{\left(2+\sqrt{3}\right)^{k+1}-\left(2-\sqrt{3}\right)^{k+1}}$ by using the first equality in  (\ref{eqn zn others}) and doing some algebra, we see that the product in \eqnref{eqn tn2} can be simplified. This gives
\begin{equation}\label{eqn tn3}
\begin{split}
t_{n}=\frac{-2\sqrt{3}}{(2+\sqrt{3})^n-(2-\sqrt{3})^n}, \qquad \text{for every $n \geq 1$},
\end{split}
\end{equation}
which can also be written as $t_n=-\frac{2 \sqrt{3} \left(2-\sqrt{3}\right)^n}{1-\left(2-\sqrt{3}\right)^{2 n}}$ for all $n \geq 1$.
Now, we turn our attention back to the solutions of $x_n$ and $y_n$. Using $x_n =t_n +y_n$, \eqnref{eqn zn recurrence2}, \eqnref{eqn tn3} and doing some algebra, the second equality in (\ref{eqn recurrence xnyn}) becomes
\begin{equation}\label{eqn yn}
\begin{split}
y_{n+1}=y_n + \frac{\sqrt{3}}{1-\left(2-\sqrt{3}\right)^{n+1}}-\frac{\sqrt{3}}{1-\left(2-\sqrt{3}\right)^n}+\frac{1}{2}, \qquad \text{for all $n \geq 1$ and $y_1=1$.}
\end{split}
\end{equation}
This can be solved as follows:
\begin{equation}\label{eqn yn2}
\begin{split}
y_n=\frac{n-2-\sqrt{3}}{2}+\frac{ \sqrt{3}}{1-\left(2-\sqrt{3}\right)^n}, \qquad \text{for all $n \geq 1$.}
\end{split}
\end{equation}
Using \eqnref{eqn yn2}, \eqnref{eqn tn3} and the fact that $x_n =t_n +y_n$, we obtain
\begin{equation}\label{eqn xn}
\begin{split}
x_n=\frac{n-2-\sqrt{3}}{2}+\frac{ \sqrt{3}}{1+\left(2-\sqrt{3}\right)^n}, \qquad \text{for all $n \geq 1$.}
\end{split}
\end{equation}
Note that for all $n \geq 1$ we have
\begin{equation}\label{eqn xnynzn}
\begin{split}
x_n+y_n-z_n&=n-1,\\
x_n-y_n+z_n&=-1-\sqrt{3}+\frac{2\sqrt{3}}{1+(2-\sqrt{3})^n},\\
-x_n+y_n+z_n&=-1-\sqrt{3}+\frac{2\sqrt{3}}{1-(2-\sqrt{3})^n}.
\end{split}
\end{equation}

Next, we obtain formulas for $r_{L_n}(p_n,p_i)$, $r_{L_n}(p_n,q_{i})$ and $r_{L_n}(p_i , q_i)$, where $n > i >1$.
We can consider $L_n$ as the union of three graphs; the upper part of $p_{i+1}$ and $q_{i+1}$, the lower part of $p_{i}$ and $q_{i}$,
and the middle part consisting of $p_{i+1}$, $q_{i+1}$, $p_{i}$ and $q_{i}$. These graphs are illustrated in \figref{fig laddergraph8}.
Note that the graphs in the upper and the lower parts are nothing but the graphs $L_{n-i}$ and $L_i$, respectively.
We make the circuit reduction of the upper part with respect to $p_n$, $p_{i+1}$ and $q_{i+1}$ to obtain a $Y$-shaped graph having the resistances $M$, $N$ and $K$ along its edges. We make the circuit reduction of the lower part with respect to $p_i$ and $q_i$.
The resistance between $p_i$ and $q_i$ in the lower part, $r_{L_i}(p_i,q_i)$, is $z_i$ by definition. Now, we have
\begin{equation}\label{eqn rpnpi}
\begin{split}
M+N =x_{n-i}, \quad
M+K =y_{n-i}, \quad
N+K =z_{n-i}.
\end{split}
\end{equation}

\begin{figure}
\centering
\includegraphics[scale=0.6]{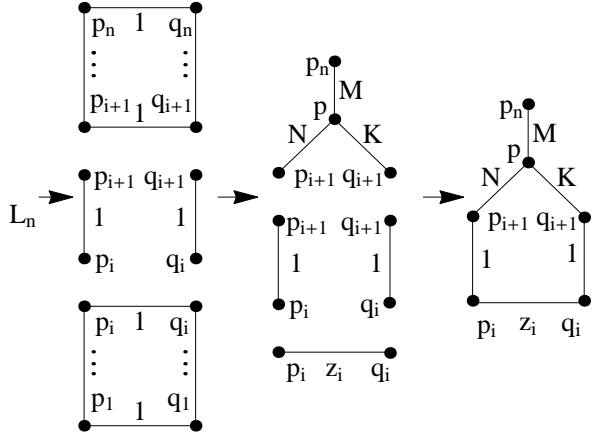} \caption{$L_n$ and circuit reductions to find $r_{L_n}(p_n , p_i)$, $r_{L_n}(q_n , p_i)$ and $r_{L_n}(p_i , q_i)$.} \label{fig laddergraph8}
\end{figure}

Solving these for $M$, $N$ and $K$, and using Equations (\ref{eqn xnynzn}) give
\begin{equation}\label{eqn rpnpi2}
\begin{split}
M &= \frac{x_{n-i}+y_{n-i}-z_{n-i}}{2}=\frac{n-i-1}{2},\\
N &= \frac{x_{n-i}-y_{n-i}+z_{n-i}}{2}=\frac{-1-\sqrt{3}}{2}+\frac{\sqrt{3}}{1+(2-\sqrt{3})^{n-i}},\\
K &= \frac{-x_{n-i}+y_{n-i}+z_{n-i}}{2}=\frac{-1-\sqrt{3}}{2}+\frac{\sqrt{3}}{1-(2-\sqrt{3})^{n-i}}.
\end{split}
\end{equation}
By making parallel and series circuit reductions in the graph at the last column of \figref{fig laddergraph8}, for each $i$ with $n >i>1$, we obtain
\begin{equation}\label{eqn rpnpi3}
\begin{split}
r_{L_n}(p_n, p_i) &= \frac{(N+1)(K+z_i+1)}{z_{n-i}+z_i+2}+M,\\
r_{L_n}(p_n, q_i) &= \frac{(K+1)(N+z_i+1)}{z_{n-i}+z_i+2}+M,\\
r_{L_n}(p_i, q_i) &= \frac{z_i(z_{n-i}+2)}{z_{n-i}+z_i+2}.
\end{split}
\end{equation}
We set
$$
\alpha=2-\sqrt{3}.
$$
Using \eqnref{eqn zn recurrence2} and Equations (\ref{eqn rpnpi2}),  we can rewrite Equations in (\ref{eqn rpnpi3}) as follows:
\begin{equation}\label{eqn rpnpi4}
\begin{split}
r_{L_n}(p_n, p_i) &= \frac{n-i}{2}+\frac{(1-\alpha^{n-i})}{4 \sqrt{3}(1-\alpha^{2n})} \big(2-2\alpha^{n+i}-\alpha^{n+i-1}-\alpha^{n-i+1}+\alpha^{2i-1}+\alpha \big), \\
r_{L_n}(p_n, q_i) &= \frac{n-i}{2}+\frac{(1+\alpha^{n-i})}{4 \sqrt{3}(1-\alpha^{2n})} \big(2+2\alpha^{n+i}+\alpha^{n+i-1}+\alpha^{n-i+1}+\alpha^{2i-1}+\alpha \big),\\
r_{L_n}(p_i, q_i) &=\frac{(1+ \alpha^{2n-2i+1})(1+ \alpha^{2i-1})}{\sqrt{3} (1-\alpha^{2n})}.
\end{split}
\end{equation}
Although we obtained formulas in $(\ref{eqn rpnpi4})$ under the condition $n>i>1$, whenever $n=i$ or $i=1$ these formulas are consistent with the ones given in Equations $(\ref{eqn zn recurrence2})$, $(\ref{eqn xn})$ and $(\ref{eqn yn2})$. Therefore, formulas in $(\ref{eqn rpnpi4})$ are valid for each integer $n$ and $i$ satisfying $n \geq i \geq 1$.

In the remaining part of this section, we obtain formulas for
$$r_{L_n}(p_i,q_j) \quad \text{and} \quad r_{L_n}(p_i,p_{j}), \quad \text{where $n > i \geq j \geq 1$}.$$
This time, we consider $L_n$ as the union of two graphs; upper and lower parts of $p_{i}$ and $q_{i}$ as illustrated in the second stage in \figref{fig ladderresij3}.
Note that the graph $L_{n-i}$ appear in the upper part, and the lower part is nothing but $L_i$ . Next, we can apply circuit reduction to reduce $L_{n-i}$ into a line with the end points $p_{i+1}$ and $q_{i+1}$, and this line has the resistance $r_{L_{n-i}}(p_{i+1},q_{i+1})=z_{n-i}$ between its end points.
For the lower part, we apply circuit reduction to $L_i$ fixing its points $p_i$, $q_i$ and $p_j$ so that we obtain a $Y$-shaped graph having the resistances $D$, $E$ and $F$ along its edges. These reductions are illustrated in the third stage in \figref{fig ladderresij3}, and the relations between $D$, $E$ and $F$ are given in Equations $(\ref{eqn rpipj})$. Finally, we obtain the reduced graph as in the last stage in \figref{fig ladderresij3}.

\begin{figure}
\centering
\includegraphics[scale=0.45]{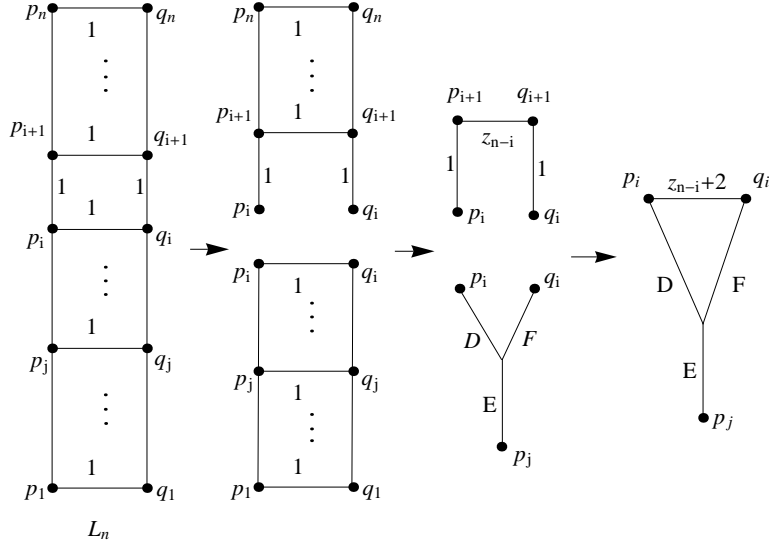} \caption{Circuit reductions applied to $L_n$ to find $r_{L_n}(p_i , p_j)$ and $r_{L_n}(p_i, q_j)$.} \label{fig ladderresij3}
\end{figure}

\begin{equation}\label{eqn rpipj}
\begin{split}
D+E =r_{L_i}(p_i,p_j), \quad
D+F =r_{L_i}(p_i,q_i)=z_{i}, \quad
E+F =r_{L_i}(q_i,p_j).
\end{split}
\end{equation}
Solving these for $D$, $E$ and $F$ gives
\begin{equation}\label{eqn rpipj2}
\begin{split}
D &= \frac{r_{L_i}(p_i,p_j)+z_{i}-r_{L_i}(q_i,p_j)}{2},\\
E &=\frac{r_{L_i}(p_i,p_j)-z_{i}+r_{L_i}(q_i,p_j)}{2},\\
F &= \frac{-r_{L_i}(p_i,p_j)+z_{i}+r_{L_i}(q_i,p_j)}{2}.
\end{split}
\end{equation}

By making parallel and series circuit reductions in the graph at the last column of \figref{fig ladderresij3}, for each $i$ with $n > i \geq j \geq 1$, we obtain
\begin{equation}\label{eqn rpipj3}
\begin{split}
r_{L_n}(p_i , p_j) &= \frac{D(z_{n-i}+F+2)}{z_{n-i}+z_i+2}+E,\\
r_{L_n}(q_i , p_j) &= \frac{F(z_{n-i}+D+2)}{z_{n-i}+z_i+2}+E,
\end{split}
\end{equation}
Now, we use \eqnref{eqn zn recurrence2}, Equations $(\ref{eqn rpnpi4})$, $(\ref{eqn rpipj2})$ and $(\ref{eqn rpipj3})$ and do some algebra using Mathematica \cite{MMA} to derive the following resistance values:
\begin{equation}\label{eqn rpipj4}
\begin{split}
r_{L_n}(p_i , p_j) &= \frac{i-j}{2}+\frac{(1-\alpha^{i-j})}{4 \sqrt{3}(1-\alpha^{2n})}\big( 2- \alpha^{i+j-1}+ \alpha^{2j-1}+\alpha^{2n-2i+1}(1-
\alpha^{i-j}-2\alpha^{i+j-1}) \big),\\
r_{L_n}(q_i , p_j) &= \frac{i-j}{2}+\frac{(1+\alpha^{i-j})}{4 \sqrt{3}(1-\alpha^{2n})}\big( 2+ \alpha^{i+j-1}+ \alpha^{2j-1}+\alpha^{2n-2i+1}(1+\alpha^{i-j}+2\alpha^{i+j-1}) \big).
\end{split}
\end{equation}
In spite of the fact that we obtained formulas in $(\ref{eqn rpipj4})$ under the condition $n > i \geq j \geq 1$, when $n=i$ these formulas are consistent with the ones given in Equations $(\ref{eqn rpnpi4})$. Therefore, formulas in $(\ref{eqn rpipj4})$ are valid for each integers $i$, $j$ and $n$ satisfying $n \geq i \geq j \geq 1$. That is, we can use the explicit formulas in $(\ref{eqn rpipj4})$ to find the resistances between any pair of vertices in $L_n$.

\section{Kirchhoff Index of $L_n$ }\label{sec Kirchhoff index}

In this section, we obtain an explicit formula for Kirchhoff index of $L_n$ by using our explicit formulas derived in \secref{sec resistances} for the resistances between any pairs of vertices of $L_n$. Moreover, we obtain an interesting summation formula by combining our findings and what is known in the literature about Kirchhoff index of $L_n$.

Recall that Kirchhoff index of a graph $\ga$, $Kf(\ga)$, is defined \cite{KR} as follows:
\begin{equation*}\label{eqn KIndex definition}
\begin{split}
Kf(\ga)=\frac{1}{2}\sum_{p,\, q \in \vv{\ga}}r(p,q).
\end{split}
\end{equation*}

\begin{theorem}\label{thm Kirchhoff index}
For any positive integer $n$, we have
$$
Kf(L_n)=\frac{n^3}{3}-\frac{n^2}{\sqrt{3}}\Big[ 1- \frac{2}{1-(2-\sqrt{3})^{2n}}\Big].
$$
\end{theorem}
\begin{proof}
With the notation of vertices as in \figref{fig laddergraphn}, using \eqnref{eqn symmetry} gives
$$Kf(L_n)= \frac{1}{2} \sum_{p,\, q \in \vv{\ga}}r(p,q)=2\sum_{1 \leq j < i \leq n}r(p_i,p_j)+
2\sum_{1 \leq j < i \leq n}r(p_i,q_j)+\sum_{i=1}^n r(p_i,q_i).
$$
Then the result follows if we use Equations $(\ref{eqn rpipj4})$ and doing some algebra \cite{MMA}.
\end{proof}
Note that the Kirchhoff index formula in \thmref{thm Kirchhoff index} can also be expressed as follows:
$$
Kf(L_n)=\frac{n^2}{3}\big[ n-\sqrt{3} \coth \big(n \ln(2-\sqrt{3}) \big) \big].
$$
The values of $Kf(L_n)$ are rational numbers. For example, its values for $1 \leq n \leq 8$ are as follows:
$1$, $5$, $\frac{71}{5}$, $\frac{214}{7}$, $\frac{11725}{209}$, $\frac{6031}{65}$, $\frac{415177}{2911}$, $\frac{140972}{679}$.

\begin{theorem}\label{thm trig sum}
For any positive integer $n$, we have
$$
\sum_{k=0}^{n-1} \frac{1}{1+2 \sin^2({\frac{\pi k}{2n}})}=\frac{n}{\sqrt{3}}\Big[ \frac{2}{1-(2-\sqrt{3})^{2n}}-1 \Big]+\frac{1}{3}.
$$
\end{theorem}
\begin{proof}
We recall the following result \cite[Theorem 4.1]{YZ} obtained by using the relation between the Kirchhoff index and the eigenvalues of the  discrete Laplacian matrix of $L_n$.
\begin{equation}\label{eqn KI Ladder}
\begin{split}
Kf(L_n)=\frac{n(n^2-1)}{3}+n \sum_{k=0}^{n-1} \frac{1}{1+2 \sin^2({\frac{\pi k}{2n}})}.
\end{split}
\end{equation}
Note that \eqnref{eqn KI Ladder} is also a particular case of \cite[Corollary 12]{CEM} (namely, when $c=1$).
Then the proof is completed by combining \eqnref{eqn KI Ladder} and the result in \thmref{thm Kirchhoff index}.
\end{proof}
Since $(2-\sqrt{3})^2 \approx  \, 0.071796$, for large values of $n$ we have $Kf(L_n) \approx \frac{n^2(n+\sqrt{3})}{3}$ by \thmref{thm Kirchhoff index}.

\begin{figure}
\centering
\includegraphics[scale=1]{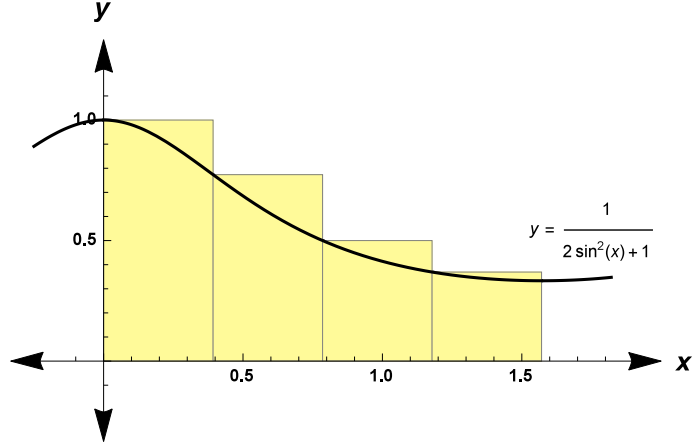} \caption{A Riemann sum approximation using left points.} \label{fig LeftRS4}
\end{figure}

Next, we give a geometric interpretation of the summation that appears in \eqnref{eqn KI Ladder}. Let $P=\{0, \, \frac{\pi}{2 n}, \, \frac{2\pi}{2 n}, \, \frac{3\pi}{2 n}, \, \cdots, \, \frac{(n-1)\pi}{2 n}, \, \frac{n\pi}{2 n}=\frac{\pi}{2}  \}$ be a partition of the interval $[0,\frac{\pi}{2}]$. Then the Riemann sum of $f(x)=\frac{1}{1+2 \sin^2(x)}$ on $[0,\frac{\pi}{2}]$ that uses left points in each subinterval is nothing but
$$
\frac{\pi}{2 n} \sum_{k=0}^{n-1} \frac{1}{1+2 \sin^2(\frac{\pi k}{2n})}.
$$
\figref{fig LeftRS4} illustrates the case with $n=4$ subintervals.
Note that
$$\displaystyle \lim_{n \rightarrow \infty} \frac{\pi}{2 n} \sum_{k=0}^{n-1} \frac{1}{1+2 \sin^2({\frac{\pi k}{2n}})} = \int_{0}^{\frac{\pi}{2}} \frac{dx}{1+2 \sin^2 x}=\frac{\pi}{2 \sqrt{3}},$$
which is consistent with our findings in \thmref{thm trig sum}.

\section{Admissible Invariants of $L_n$ }\label{sec admissible constants}

In this section, we give explicit formulas for the following admissible invariants of $L_n$:
$\tau(L_n)$, $\theta(L_n)$, $\lambda(L_n)$, $\varphi(L_n)$ and $\epsilon(L_n)$ when $L_n$ is considered as a model of a metrized graph.
These invariants were studied in \cite{C2}, \cite{C3}, \cite{C4}, \cite{Zh} and the references therein.

\begin{theorem}\label{thm theta}
For any positive integer $n$, we have
$$\theta(L_n)=\frac{2(n-2)}{3} \Big[ n^2-4n+10-(n-6) \sqrt{3}(1-\frac{2}{1-(2-\sqrt{3})^{2n}})  \Big].$$
\end{theorem}
\begin{proof}
By definition \cite[Section 4]{C2}, $\theta(L_n)=\sum_{p, \, q \in \vv{L_n}} (\va(p)-2)(\va(q)-2)r(p,q)$, where $\va(p)$ is the degree of the vertex $p$. Therefore,
\begin{equation*}\label{eqn theta and Kirchhoff}
\begin{split}
2 Kf(L_n) - \theta(L_n) & =8 \sum_{p \,  \in  \vv{L_n}}r(p,p_n)-4\big[ r(p_n,p_1)+r(p_n,q_1)+r(p_n,q_n) \big] \\
&=8 \sum_{i=2}^{n-1}\big( r(p_i,p_n)+r(q_i,p_n) \big)+4\big[ r(p_n,p_1)+r(p_n,q_1)+r(p_n,q_n) \big],
\end{split}
\end{equation*}
where the vertices $p_i$ and $q_i$ with $i \in \{ 1, \, \dots, \, n \}$ are as in \figref{fig laddergraphn}. Thus, the result follows by using this equality, Equations $(\ref{eqn rpipj4})$, \thmref{thm Kirchhoff index} and doing some algebra \cite{MMA}.
\end{proof}

\begin{theorem}\label{thm tau}
For any positive integer $n$, we have
$$\tau(L_n)=\frac{9n -20}{36}+ \frac{n-6}{6 \sqrt{3}} \Big[ 1-\frac{2}{1-(2-\sqrt{3})^{2n}} \Big].$$
\end{theorem}
\begin{proof}
Let $\ga$ be a graph with set of vertices $\vv{\ga}$ such that each edge length in $\ga$ is $1$. Suppose $\ga$ is a model of a metrized graph.
If we use \cite[Proposition 2.6]{C3}, \cite[Equation (3)]{C3} and \cite[Proof of Lemma 4.9]{C3}, we obtain the following formula of the tau constant $\tau(\ga)$ of $\ga$ for every $s \in \vv{\ga}$:
\begin{equation}\label{eqn tau general}
\begin{split}
\tau(\ga)=\frac{1}{12} \sum_{\substack{p \sim q \\ p, \, q \in \vv{\ga}} } (1-r(p,q))^2
+\frac{1}{4} \sum_{\substack{p \sim q \\ p, \, q \in \vv{\ga}} } \big( r(s,p)-r(s,q) \big)^2,
\end{split}
\end{equation}
where $p \sim q$ means $p$ and $q$ are adjacent, i.e., connected by an edge in $\ga$.

Using the notations in \figref{fig laddergraphn} and the symmetry in $L_n$, we can rewrite \eqnref{eqn tau general} for $L_n$ with $s=p_n$ as follows:
\begin{equation*}\label{}
\begin{split}
\tau(L_n) &=\frac{1}{6} \sum_{i=1}^{n-1} (1-r(p_i,p_{i+1}))^2+ \frac{1}{12} \sum_{i=1}^{n} (1-r(p_i,q_{i}))^2
+\frac{1}{4} \sum_{i=1}^{n-1} \big( r(p_n,p_{i})-r(p_n,p_{i+1}) \big)^2\\
& \quad +\frac{1}{4} \sum_{i=1}^{n} \big( r(p_n,p_{i})-r(p_n,q_{i}) \big)^2
+\frac{1}{4} \sum_{i=1}^{n-1} \big( r(p_n,q_{i})-r(p_n,q_{i+1}) \big)^2.
\end{split}
\end{equation*}
Therefore, the proof follows if we use this equality, Equations $(\ref{eqn rpipj4})$ and doing some algebra \cite{MMA}.
\end{proof}

\begin{theorem}\label{thm admissible inv}
For any positive integer $n$, we have
\begin{equation*}\label{}
\begin{split}
& \varphi(L_n)=\frac{3n^3-9n^2-5n+1}{18(n-1)}+\frac{(n-6)(2n-1)}{6\sqrt{3}(n-1)}\Big[ 1-\frac{2}{1-(2-\sqrt{3})^{2n}} \Big],\\
& \lambda(L_n)=\frac{n(n+4)(n-1)}{12(2n-1)},\\
& \epsilon(L_n)=\frac{(3n^2-3n+10)(n-2)}{9(n-1)}-\frac{(n-2)(n-6)}{3 \sqrt{3} (n-1)} \Big[ 1-\frac{2}{1-(2-\sqrt{3})^{2n}} \Big],\\
& Z(L_n)= \frac{(3n^2-13)n}{36(n-1)^2}+ \frac{n(n-6)}{12 \sqrt{3} (n-1)^2} \Big[ 1-\frac{2}{1-(2-\sqrt{3})^{2n}} \Big].
\end{split}
\end{equation*}
\end{theorem}
\begin{proof}
Since each of $\varphi(L_n)$, $\lambda(L_n)$, $\epsilon(L_n)$ and $Z(L_n)$ can be expressed in terms of $\tau(L_n)$, $\theta(L_n)$ and $\ell(L_n)$
\cite[Propositions 4.6, 4.7, 4.8 and 4.9]{C2} with $g(L_n)=(3n-2)-(2n)+1=n-1$, the results follow from
\thmref{thm theta} and \thmref{thm tau}.
\end{proof}

Note that our findings in this section are consistent with the numeric results given in \cite[Table 5]{C4} for $n \in \{5, \, 10, \, 15, \, 20 \}$.

Finally, we observe the following behavior of these invariants:
\begin{align*}
& \lim_{n->\infty} \frac{\tau(L_n)}{\ell(L_n)}=\frac{1}{108}(9-2\sqrt{3}),  & \lim_{n->\infty} \frac{Z(L_n)}{\ell(L_n)}=\frac{1}{36},\\
&\lim_{n->\infty} \frac{1}{g(L_n)}\frac{\varphi(L_n)}{\ell(L_n)}=\frac{1}{18},  &\lim_{n->\infty} \frac{1}{g(L_n)}\frac{\epsilon(L_n)}{\ell(L_n)}=\frac{1}{9},\\
&\lim_{n->\infty} \frac{1}{g(L_n)}\frac{\lambda(L_n)}{\ell(L_n)}=\frac{1}{72},  &\lim_{n->\infty} \frac{1}{g^2(L_n)} \frac{\theta(L_n)}{\ell(L_n)}=\frac{2}{9}.
\end{align*}

\section{Connection to Generalized Fibonacci Numbers}\label{sec Gen Fib}

We note that the powers of $2-\sqrt{3}$ appear in the binet formula of certain generalized Fibonacci numbers \cite{KM}. Namely, for the sequence of integers $G_n$ defined by the following recurrence relation
$$G_{n+2}=4G_{n+1}-G_{n}, \quad \text{if $n \geq 2$, and $G_0=0$, $G_1=1$},$$
we have
$$
G_{n}=\frac{(2-\sqrt{3})^{-n}-(2-\sqrt{3})^n}{2 \sqrt{3}}, \quad \text{for each integer $n \geq 0$}.
$$
The values of $G_n$ with $0 \leq n \leq 10$ are as follows:
$0$,  $1$, $4$,  $15$, $56$, $209$, $780$, $2911$, $10864$, $40545$, $151316$.

Various properties of the sequence $G_n$ are well-known in the literature \cite{SL}.
For example, we recognize the number $G_n$ as the number of spanning trees of $L_n$ \cite{BP}.

Since we have
$$
(2-\sqrt{3})^n=\frac{1}{G_{n+1}-(2-\sqrt{3})G_n}, \quad \text{for each integer $n \geq 0$}
$$
and
$$
G_{2n}=G_n \big( (2-\sqrt{3})^{-n}+(2-\sqrt{3})^n  \big)=-2 \sqrt{3} G_{n}^2 \coth{(n \ln{(2-\sqrt{3})})},
$$
we can rewrite our findings in the previous sections in terms of $G_n$. Namely, we obtained the following results in this paper:

For every integer $n \geq 1$,
\begin{align*}
t_n &=-\frac{1}{G_n},&  z_n&=-1+\frac{G_{2n}}{2G_n^2},
\\ y_n&=\frac{n-2}{2}+3\frac{G_n^2}{G_{2n}-2G_n},&  x_n&=\frac{n-2}{2}+\frac{G_{2n}-2G_n}{4G_n^2}.
\end{align*}

If we let $g_n:=\frac{1}{G_{n+1}-(2-\sqrt{3})G_n}$, we can rewrite \eqnref{eqn rpipj4} in the following form
%
\begin{equation*}\label{eqn rpipj4 third}
\begin{split}
r_{L_n}(p_i , p_j) &= \frac{i-j}{2}+\frac{1-g_{i-j}}{8 \sqrt{3}}(1+\frac{G_{2n}}{2 \sqrt{3} G_n^2})
\Big[ (1-g_{i+j-1})(1+g_{2n-2i+1})\\
& \quad + (1+g_{2j-1})(1-g_{2n-i-j+1})\Big]\\
r_{L_n}(q_i , p_j) &= \frac{i-j}{2}+\frac{1+g_{i-j}}{8 \sqrt{3}}(1+\frac{G_{2n}}{2 \sqrt{3} G_n^2})
\Big[ (1+g_{i+j-1})(1+g_{2n-2i+1})\\
& \quad + (1+g_{2j-1})(1+g_{2n-i-j+1})\Big],\\
\end{split}
\end{equation*}
where $n \geq i \geq j \geq 1$.

Here is how we can express the results given in \thmref{thm Kirchhoff index} and \thmref{thm trig sum} in terms of $G_n$:
\begin{equation}\label{eqn KI Ladder II}
\begin{split}
Kf(L_n)=\frac{n^3}{3}+ \frac{n^2 G_{2n}}{6 G_{n}^2}, \qquad \text{and } \quad
\sum_{k=0}^{n-1} \frac{1}{1+2 \sin^2{(\frac{k \pi}{2n})}}=\frac{1}{3}+ \frac{n G_{2n}}{6 G_{n}^2}.
\end{split}
\end{equation}
If $\lambda_1, \, \lambda_2, \, \dots, \, \, \lambda_m$ are nonzero eigenvalues of a connected graph $\Gamma$ with $m$ vertices, then
$Kf(\Gamma)=m \sum_{i=1}^m \frac{1}{\lambda_i}$ (\cite{GM} and \cite{ZK}).
Since $2$, $2-2\cos{(\frac{k \pi}{n})}=4\sin^2{(\frac{k \pi}{2n})}$ and $4-2\cos{(\frac{k \pi}{n})}=2+4 \sin^2{(\frac{k \pi}{2n})}$
for $k=1, \, 2, \, \dots, \, n-1$ are the nonzero eigenvalues of the discrete Laplacian matrix of $L_n$ \cite[Proof of Theorem 6]{BP}, we have
\begin{equation}\label{eqn KI Ladder III}
\begin{split}
Kf(L_n)=n+ \frac{n}{2} \sum_{k=1}^{n-1} \frac{1}{ \sin^2{(\frac{k \pi}{2n})}}+
n \sum_{k=1}^{n-1} \frac{1}{1+2 \sin^2{(\frac{k \pi}{2n})}}.
\end{split}
\end{equation}
Then the following equality follows from Equations $(\ref{eqn KI Ladder II})$ and \eqnref{eqn KI Ladder III},
\begin{equation}\label{eqn KI Ladder IV}
\begin{split}
 \sum_{k=1}^{n-1} \frac{1}{ \sin^2{(\frac{k \pi}{2n})}}=\frac{2(n^2-1)}{3}.
\end{split}
\end{equation}

Since the Chebyshev polynomial of the second kind $U_n(x)$ is given by the relation $U_{n+2}(x)=2 x U_{n-1}(x)-U_n(x)$ for $n \geq 0$ and the initial values $U_1(x)=1$ and $U_0(x)=0$, we have $G_n = U_{n-1}(2)$. That is, the formulas we found are nothing but expressions involving Chebyshev polynomials. Therefore, combining the formulas in \eqnref{eqn KI Ladder II} with  the ones given in \cite[Corollary 12]{CEM} (when $a=c=1$), we obtain the following equality:
$$6U'_{n-1}(2)=n \frac{U_{2n-1}(2)}{U_{n-1}(2)}-4U_{n-1}(2).$$

Next, we express the admissible invariants of $L_n$ in terms of the numbers $G_n$:
\begin{align*}
\tau(L_n)&=\frac{9n -20}{36}- \frac{(n-6)G_{2n}}{36 G_n^2},&  \theta(L_n)&=\frac{2(n-2)}{3} \Big[ n^2-4n+10+ \frac{(n-6)G_{2n}}{2G_n^2} \Big],
\\  \lambda(L_n)&=\frac{n(n+4)(n-1)}{12(2n-1)},&  \varphi(L_n)&=\frac{3n^3-9n^2-5n+1}{18(n-1)}-\frac{(n-6)(2n-1)G_{2n}}{36(n-1) G_n^2},
\end{align*}
and similarly we have
\begin{equation*}\label{}
\begin{split}
& \epsilon(L_n)=\frac{(n-2)}{9 (n-1)} \Big[ 3n^2-3n+10+\frac{(n-6)G_{2n}}{2G_n^2}\Big],\\
& Z(L_n)= \frac{n}{36 (n-1)^2} \Big[3n^2-13- \frac{(n-6)G_{2n}}{2 G_n^2} \Big].
\end{split}
\end{equation*}


\textbf{Acknowledgements:} This work is supported by The Scientific and Technological Research Council of Turkey-TUBITAK (Project No: 110T686).

\end{document}